\documentclass[12pt]{article}
\usepackage[utf8]{inputenc}
\usepackage{amsthm}
\usepackage{amsmath}
\usepackage{amsfonts}
\usepackage{graphicx}
\usepackage{latexsym}
\usepackage{amssymb}
\usepackage{array}
\usepackage{upref}

\usepackage{mathrsfs}

 \usepackage[colorlinks=true]{hyperref}
  \hypersetup{urlcolor=blue, citecolor=red}

\DeclareMathAlphabet{\mathbfsl}{OT1}{ppl}{b}{it} %{OT1}{cmr}{bx}{it}

\newcommand{\mmod}{{\mbox{mod}}}

\makeatletter
 \DeclareRobustCommand{\nsbinom}{\genfrac[]\z@{}}
 \makeatother

\newcommand{\abs}[1]{\left|#1\right|}

\newcommand{\field}[1]{\mathbb{#1}}

\newcommand{\Z}{\field{Z}}

\newcommand{\cB}{{\cal B}}

\newcommand{\cD}{{\cal D}}

\newcommand{\cF}{{\cal F}}

\newcommand{\cS}{{\mathcal{S}}}

\newcommand{\linadd}{\kern1pt\mbox{\small$\boxplus$}\kern1pt}

\newcommand{\Cref}[1]{Co\-ro\-lla\-ry\,\ref{#1}}

\newtheorem{theorem}{Theorem}[section]

\newtheorem{lemma}[theorem]{Lemma}
\newtheorem{corollary}[theorem]{Corollary}
\newtheorem{example}[theorem]{Example}
\newtheorem{construction}[theorem]{Construction}
\newtheorem{remark}[theorem]{remark}

\title{Pairs in Nested Steiner Quadruple Systems}

\date{\today}
\author{\Large {\textbf{Yeow Meng Chee$^\text{x}$}, \textbf{Son Hoang Dau$^\text{+}$},}\\
{\Large \textbf{Tuvi Etzion$^{\text{x}*}$}, \textbf{Han Mao Kiah$^\bullet$}, \textbf{Wenqin Zhang$^\dag$}}\\
{\small $^\text{x}$Dept. of Industrial Systems Engineering and Management, National University of Singapore}\\
{\small $^+$ School of Computing Technologies, RMIT University, Melbourne, VIC 3000, Australia}\\
{\small $^*$Computer Science Faculty, Technion, Israel Institute of Technology, Haifa 3200003, Israel}\\
{\small $^\bullet$School of Physical and Mathematical Sciences, Nanyang Technological University, Singapore}\\
{\small $^\dag$School of Electronic Information and Engineering, Shanghai Jiao tong University, China}\\
{\small {\it ymchee@nus.edu.sg}, {\it sonhoang.dau@rmit.edu.au}, {\it etzion@cs.technion.ac.il},}\\
{\small {\it hmkiah@ntu.edu.sg}, {\it wenqin$\_$zhang@sjtu.edu.cn}   }
}

\begin{document}

\maketitle

\begin{abstract}
Motivated by a repair problem for fractional repetition codes in distributed storage, each block
of any Steiner quadruple system (SQS) of order $v$ is partitioned into two pairs.
Each pair in such a partition is called a nested design pair
and its multiplicity is the number of times it is a pair in this partition.
Such a partition of each block is considered as a new block design called a nested Steiner quadruple system.
Several related questions on this type of design are considered in this paper:
What is the maximum multiplicity of the nested design pair with minimum multiplicity?
What is the minimum multiplicity of the nested design pair with maximum multiplicity?
Are there nested quadruple systems in which all the nested design pairs have the same multiplicity?
Of special interest are nested quadruple systems in which
all the $\binom{v}{2}$ pairs are nested design pairs with the same multiplicity.
Several constructions of nested quadruple systems are considered and in particular classic constructions of SQS are examined.
\end{abstract}

%\footnotetext[1] { This research was supported in part by the Israeli
%Science Foundation (ISF), Jerusalem, Israel, under
%Grant 10/12.}

%%%%%%%%%%%%%%%%%%%%%%%%%%%%%%%%%%%%%%%%%%%%%%%%%%%%%%%%%%%%%%%%%%%%%%
%%%%%%%%%%%%%%%%%%%%%%%%%%%%%%%%%%%%%%%%%%%%%%%%%%%%%%%%%%%%%%%%%%%%%%
%%%%%%%%%%%%%%%%%%%%%%%%%%%%%%%%%%%%%%%%%%%%%%%%%%%%%%%%%%%%%%%%%%%%%%
%\newpage
\section{Introduction}
\label{sec:intro}

A {\bf \emph{Steiner quadruple system}} (SQS) is a pair $(Q,\cB)$, where $Q$ is a finite set of {\bf \emph{points}}
and $\cB$ is a collection of 4-subsets of $Q$ called {\bf \emph{blocks}} such that every 3-subset of $Q$
is contained in exactly one block of $\cB$.
Motivated by a repair problem for fractional repetition codes~\cite{CDEKLZ} we ask natural questions associated with a partition
of each block in an SQS of order $v$ (SQS$(v)$), where $v$ is the number of points in $Q$, into pairs.
An SQS of order $v$ with such a partition will be called a {\bf \emph{nested Steiner quadruple system}}
(in short, {\bf \emph{nested SQS}} or {\bf \emph{nested SQS$(v)$}}).
For a given SQS and a nested quadruple system for it, the number of times ({\bf \emph{multiplicity}}) that each pair is contained
in this partition is considered. A pair which is contained in the partition
of some block will be called a {\bf \emph{nested design pair}} (in short, {\bf \emph{ND-pair}}) .
Hanani~\cite{Han60} proved that an SQS$(v)$ exists if and only if $v \geq 4$
and ${v \equiv 2 ~ \text{or} ~ 4 ~(\mmod ~ 6)}$. His proof consists of six recursive constructions which together cover all the values of
$v$ that are congruent to 2 or 4 modulo~6. Other recursive constructions for SQSs were given by~\cite{Han63,Rok67}.
By carefully examining some of these constructions interesting nested quadruple systems can be found.
The nested quadruple system problem is motivated by a repair problem for fractional repetition codes, where for a block $\{ x,y,z,w \}$
three pairs of the form $\{ x,y \}$, $\{ y,z \}$, and $\{ z,w \}$ are considered~\cite{CDEKLZ}.
Nested designs were considered in the past, e.g.~\cite{CFM20,FKKMS,KuFu94,OJKM} and references therein under the names
of nested designs or split-block designs. In these papers there are also partitions of the elements in the blocks,
but to our best knowledge, the current paper is the first one which considers such a partition of each block
into pairs and an analysis of the ND-pairs.

In this work we consider only the combinatorial design problems, where the block is partitioned into two pairs
such as $\{ x,y \}$ and $\{ z,w \}$, and we concentrate on a few related questions:

\begin{enumerate}
\item[(1)] For a given $v$, does there exist a nested SQS$(v)$ in which the multiplicity
of all the ND-pairs is the same?

\item[(2)] For a given $v$, does there exist a nested SQS$(v)$ in which all pairs
are ND-pairs and the multiplicity of each one is the same?

\item[(3)] For a given $v$, is there a threshold $\mu(v)$ such that there exists a nested SQS$(v)$
in which the multiplicity of each ND-pair is at least $\mu(v)$?

\item[(4)] For a given $v$, what is the threshold $\mu (v)$, for which there exists an SQS$(v)$, for which all the $\binom{v}{2}$ pairs
are ND-pairs whose multiplicity is at least $\mu(v)$?

\item[(5)] Is there an upper bound on the multiplicity of the ND-pair with the smallest multiplicity?
Is there a lower bound on the multiplicity of the ND-pair with the largest multiplicity?

\item[(6)] What is the minimum number of ND-pairs in a nested SQS$(v)$?

\item[(7)] Finally, we want to propose constructions for nested quadruple systems for which we can analyze the number of
ND-pairs and their multiplicities.
\end{enumerate}

Note that the difference between (1) and (2) to (3) and (4) is that in (1) and (3) not all the pairs are ND-pairs, while in
(2) and (4) all the pairs are ND-pairs.

\begin{example}
\label{ex:Bool8}
We start with a nested SQS$(8)$ which answer some questions for $v=8$.

Let $\Z_8=\{0,1,2,3,4,5,6,7\}$ be the set of points and consider the following set of pairs for a nested SQS$(8)$.

$$
[ \{0,1 \},\{2,3\}], ~ [ \{0,1 \},\{4,5\}], ~ [ \{0,1 \},\{6,7\}], ~ [ \{0,2 \},\{4,6\}], ~ [ \{0,2 \},\{5,7\}],
$$
$$
[ \{0,3 \},\{4,7\}], ~[ \{0,3 \},\{5,6\}], ~[ \{1,2 \},\{5,6\}], ~[ \{1,2 \},\{4,7\}],  ~[ \{1,3 \},\{4,6\}],
$$
$$
[ \{1,3 \},\{5,7\}], ~ [ \{2,3 \},\{4,5\}], ~[ \{2,3 \},\{6,7\}], ~[ \{4,5 \},\{6,7\}]~.
$$

These 14 blocks form an SQS(8), where each one of the ND-pairs $\{0,1\}$, $\{2,3\}$, $\{4,5\}$, and $\{6,7\}$ has multiplicity 3.
Each one of the ND-pairs $\{ 0,2\}$, $\{1,3\}$, $\{4,6\}$, $\{5,7\}$, $\{0,3\}$, $\{1,2\}$, $\{ 4,7\}$, and $\{5,6\}$ has multiplicity 2.

Consider this representation of the unique SQS$(8)$~\cite{Bar08}, where we can consider the integer $k \in \Z_8$ with its binary representation as
a vector of length 3. We can observe that the sum of the points in each block is the all-zeros vector and hence
this SQS form the words of length~8
in the extended Hamming code. We can also observe that the 12 ND-pairs represent all the pairs $\{ x,y \}$, where $x+y$
is equal to $001$, $010$, or $011$.

\hfill\quad $\blacksquare $
\end{example}

The constructions which will be given in the sequel will use some combinatorial designs and especially certain types of
SQSs. We summarize two types of such SQSs now.

\noindent
{\bf \underline{Rotational SQS}:}

A rotational SQS$(v+1)$ is a system $(\Z_v \cup \{\infty\},\cB)$, where the blocks have the following property.
If the block $\{ \infty,x,y,z \}$ is in~$\cB$, then also the block $\{ \infty,x+1,y+1,z+1 \}$ is in~$\cB$ and if the block
$\{ x,y,z,w \}$, that does not contain the point $\infty$, is in $\cB$, then also the block $\{ x+1,y+1,z+1,w+1 \}$ is in $\cB$.
The point $\infty$ is a fixed point, while on all the other points cyclic shifts associated with $\Z_v$ are performed from
one block to the next block. In the nested quadruple system which will
be discussed, if the block $\{ \infty,x,y,z \}$ is partitioned into the two pairs $\{ \infty,x\}$ and $\{y,z \}$, then the block
$\{ \infty,x+1,y+1,z+1 \}$ is partitioned into the two pairs $\{ \infty,x+1\}$ and $\{y+1,z+1 \}$.
If the block $\{ x,y,z,w \}$ is partitioned into the two pairs $\{x,y\}$ and $\{z,w \}$, then the block
$\{ x+1,y+1,z+1,w+1 \}$ is partitioned into the two pairs $\{ x+1,y+1\}$ and $\{z+1,w+1 \}$.
This will enable us to analyze the nested SQS only by a set of base blocks,
i.e., consider the orbits of the system. Some work on rotational SQS can be found in~\cite{Car56,FCJ08,JiZh02,Phe77}.
Another type of rotational SQS$(v+1)$ has a {\bf \emph{multiplier group}} $M$ which contains a subgroup of the
multiplicative group of the integers modulo $v$.
If $\alpha \in M$ and $\{ x,y,z,w \}$ is a block in the system, then also $\{ \alpha x,\alpha y, \alpha z,\alpha w \}$ in a block in the system.
Now, if the block $\{ x,y,z,w \}$ is partitioned into the two pairs $\{x,y\}$ and $\{z,w \}$, then the block
$\{ \alpha x,\alpha y, \alpha z,\alpha w \}$ is partitioned into the two pairs $\{\alpha x, \alpha y\}$ and $\{ \alpha z, \alpha w \}$.
This enable us to reduce further the set of base blocks for the system.

\vspace{0.2cm}

\noindent
{\bf \underline{Boolean SQS}:}

Boolean SQS$(2^n)$ is a system $(\Z_2^n,\cB)$, where the quadruple $\{ x, y, z, w \}$ is a block in $\cB$ if and only if $x+y+z+w$ is
the all-zeros vector of length $n$. Example~\ref{ex:Bool8} is a nested SQS$(8)$.
The system can be represented  differently, where each nonzero point $x$ is a power of
a primitive element $\alpha$ in GF$(2^n)$. This representation implies that if $\{ x,y,z,w \}$ is a block in~$\cB$, then also
$\{ \alpha x, \alpha y , \alpha z , \alpha w \}$ is a block in $\cB$. Hence, if the points are ordered as consecutive
power of $\alpha$, by this representation the Boolean SQS$(2^n)$ is
a rotational SQS$(2^n)$, where the all-zeros vector is the $\infty$ fixed point.

\vspace{0.2cm}

A block from which other blocks are generated for a rotational SQS or a Boolean SQS is called a {\bf \emph{base block}}.

The rest of the paper is organized as follows.
In Section~\ref{sec:large_int} two questions on the multiplicity of ND-pairs are considered.
First, we consider a lower bound on the number of ND-pairs in a nested SQS.
Then, we find an upper bound on the multiplicity of the ND-pair with the smallest multiplicity
and a lower bound on the multiplicity of the ND-pair with the largest multiplicity.
Other bounds on the number of ND-pairs and their multiplicities will be considered too.
In Section~\ref{sec:recursive} we introduce two recursive constructions
for a nested SQS in which each ND-pair has a relatively large multiplicity.
These constructions are doubling constructions, i.e., from a nested SQS$(v)$ we construct a nested SQS$(2v)$. In these constructions
the multiplicity of the ND-pairs in the nested SQS$(v)$ will be taken into account when the multiplicity of
the pairs in the blocks nested SQS$(2v)$ will be calculated. But, in any case the multiplicity of the ND-pairs will be large.
Section~\ref{sec:uniform} discusses whether there exist nested quadruple systems in which all the
ND-pairs have the same multiplicity. We prove that such a nested SQS$(v)$ in which all the pairs
are ND-pairs can exist only if $v \equiv 2 ~(\mmod ~ 6)$. We prove the existence of such uniform nested quadruple systems for
some values and suggest some construction methods. We consider such uniform nested quadruple systems when the number
of ND-pairs is the smallest possible. Finally, in Section~\ref{sec:conclude}
a summary of the results is given with a few interesting questions which we have not been able to solve.

\section{Bounds Concerning the Nested Design Pairs}
\label{sec:large_int}

In this section we derive a lower bound on the number of ND-pairs in a nested SQS.
We continue and prove an upper bound on the multiplicity of ND-pairs with the smallest multiplicity.
In Section~\ref{sec:recursive} we will prove that these two bounds are tight.
The first lemma is an immediate consequence from the fact that each triple is contained in exactly one block of an SQS.

\begin{lemma}
\label{lem:max_multiplicity}
Any pair $\{x,y\}$ in a block of an SQS$(v)$, is contained in $\frac{v-2}{2}$ blocks, i.e., the multiplicity of
an ND-pair in the nested SQS is at most $\frac{v-2}{2}$.
\end{lemma}

\begin{lemma}
\label{lem:maxPairsMax}
The number of ND-pairs with maximum multiplicity $\frac{v-2}{2}$ in a nested SQS$(v)$ is at most $\frac{v}{2}$.
\end{lemma}
\begin{proof}
Let $\{x,y\}$ be a ND-pair with multiplicity $\frac{v-2}{2}$.
Consider first the multiplicity of a pair $\{x,z\}$, where $z \neq y$.
Let $\{x,z,y,w\}$ be the unique quadruple which contains the triple $\{x,y,z\}$.
Since $\{x,y\}$ has multiplicity $\frac{v-2}{2}$, it follows that the block $\{x,z,y,w\}$ is partitioned
into $\{x,y\}$ and $\{z,w\}$ in the nested SQS and hence $\{x,z\}$ does not have multiplicity $\frac{v-2}{2}$.
%Assume now that a pairs $\{ z,a\}$ has also multiplicity $\frac{v-2}{2}$, where $\{ x,y \} \cap \{z,a\} =\varnothing$.
%Since $\{ x,y \}$ also has multiplicity $\frac{v-2}{2}$, it follows that $\{x,y,z,a\}$ is the unique block which contains
%the triple $\{x,y,z\}$.
Applying iteratively these arguments implies that all ND-pairs with multiplicity $\frac{v-2}{2}$
are pairwise disjoint and hence the claim of the lemma follows.
\end{proof}

Since the size of an SQS$(v)$ is $\frac{v(v-1)(v-2)}{24}$ and each block is partitioned into two pairs for the nested SQS,
we have the following consequence.

\begin{lemma}
\label{lem:total_pairs}
The total number of pairs in the blocks of a nested SQS$(v)$ is $\frac{v(v-1)(v-2)}{12}$.
\end{lemma}

\begin{lemma}
\label{lem:lower_nonzeroPN}
If $(Q,\cB)$ is a nested SQS$(v)$, then each element of $Q$, is contained in at least $\frac{v-2}{2}$ ND-pairs.
\end{lemma}
\begin{proof}
Assume $x$ is an element of $Q$. Let $Q_1$ be the set of elements of $Q$ that are paired with~$x$ for ND-pairs
of the nested SQS. Assume that $Q_1$ has $k$ elements, where $k < v-1$. Let $Q_2$ be the set of $v-k-1$ elements
of $Q \setminus \{x\}$ that are not paired with $x$ for an ND-pair and let $y \in Q_2$.
%Let $y \in Q_2$ and consider the $v-k-2$ blocks of the form
%$$
%\{ \{ x,y,z,w \} ~:~ z \in Q_2, ~ w \in Q_1 \}.
%$$
%Since each triple of $Q$ is contained in one block of $\cB$, it follows that the pairs $\{ z,w \}$ in these blocks must
%be disjoint. For each $z \in Q_2$ there exists such a block and hence $k \geq v-k-2$, i.e., $k \geq \frac{v-2}{2}$.
There are $\frac{v-2}{2}$ blocks with the pair $\{x,y\}$ and they contribute $\frac{v-2}{2}$ distinct ND-pairs
that contain $x$. Thus, $k \geq \frac{v-2}{2}$.
\end{proof}

%We note that in the proof of Lemma~\ref{lem:lower_nonzeroPN} we cannot use a similar argument to show that $v-k-2 \geq k$.
%The reason for this is that a block can contain $x$, $y$ and two more elements from $Q_1$, but it cannot contain two elements
%from $Q_2$ since then the block will contain three elements that are not paired with $x$ for a ND-pair, which is impossible.

Lemma~\ref{lem:lower_nonzeroPN} also implies a simple lower bound on the number of ND-pairs in a nested SQS.
\begin{corollary}
\label{cor:lower_nonzeroPN}
The number of ND-pairs in a nested SQS$(v)$ is at least $\frac{v}{2} \left( \frac{v}{2} -1 \right)$.
\end{corollary}
\begin{proof}
By Lemma~\ref{lem:lower_nonzeroPN} each element of $Q$ is contained in at least $\frac{v-2}{2}$ ND-pairs of the nested SQS.
There are $v$ elements in $Q$ for a total of at least $\frac{v}{2}\frac{v-2}{2}$ ND-pairs, where the division by two is done since each
pair is counted twice.
\end{proof}
The upper bounds of Lemmas~\ref{lem:max_multiplicity} and~\ref{lem:maxPairsMax} and also
the lower bounds of Lemma~\ref{lem:lower_nonzeroPN}
and Corollary~\ref{cor:lower_nonzeroPN} are attained by constructions presented in Section~\ref{sec:recursive}.
But, there are values of $v$ for which the lower bounds cannot be attained as a consequence of the next necessary
condition for the existence of a nested SQS$(v)$ with $\frac{v}{2} \left( \frac{v}{2} -1 \right)$ ND-pairs.

\begin{lemma}
\label{lem:necc_min_NZ}
If $\cS=(Q,\cB)$ is a nested SQS$(v)$ for which there are
$\frac{v}{2} \left( \frac{v}{2} -1 \right)$ ND-pairs, then the set
of points $Q$ can be partitioned into two subsets $Q_1$ and $Q_2$ of equal size $\frac{v}{2}$ such that
all pairs of $Q_1$ and all pairs of $Q_2$ are the only ND-pairs.
\end{lemma}
\begin{proof}
Let $x$ and $y$ be two distinct points for which $\{ x,y \}$ is not a ND-pair. There are $\frac{v-2}{2}$
blocks of the form $\{ x,y,a_i,b_i \}$, $1 \leq i \leq \frac{v-2}{2}$, where $\{ a_i , b_i \} \cap \{ a_j , b_j \} = \varnothing$
for $1 \leq i<j \leq \frac{v-2}{2}$. By Lemma~\ref{lem:lower_nonzeroPN} $x$ is contained in at least $\frac{v-2}{2}$ ND-pairs.
But, since the number of ND-pairs meets the lower bound of Corollary~\ref{cor:lower_nonzeroPN}, it follows that
$Q$ can be partitioned into to sets $Q_x$ and $Q_y$, each one of size $\frac{v}{2}$, such that
$x \in Q_x$, $y \in Q_y$, for each $a \in Q_x \setminus \{ x \}$ we have that $\{x,a\}$ is a ND-pair, and
for each $b \in Q_y \setminus \{ y \}$ we have that $\{y,b\}$ is a ND-pair. Moreover,
for each $a \in Q_x$ we have that $\{y,a\}$ is not a ND-pair and for each $b \in Q_y$ we have that $\{x,b \}$ is not a ND-pair.
Hence, the same arguments applied on the pair $\{ x,y \}$ can be applied on the pairs $\{x,b\}$ and $\{y,a\}$ and in both
cases $Q$ is partitioned to exactly the same two subsets $Q_x$ and $Q_y$.
This implies the claim of the lemma.
\end{proof}

\begin{lemma}
\label{lem:necc_min_NZ_odd}
If $v \equiv 2 ~ \text{or} ~ 10 ~(\mmod ~12)$, then in each nested SQS$(v)$, on $\Z_v$,
each element is contained in at least $\frac{v}{2}$ ND-pairs.
\end{lemma}
\begin{proof}
Assume for the contrary that there exists an element $x \in \Z_v$ which is contained in only $\frac{v}{2}-1$ ND-pairs
(by Lemma~\ref{lem:lower_nonzeroPN} it cannot be smaller).
Let $Q_1$ be the set which contains $x$ and each element $y \in \Z_v$ such that $\{ x,y \}$ is an ND-pair.
Since $x$ is contained in $\frac{v}{2}-1$ ND-pairs, it follows that $\abs{Q_1} = \frac{v}{2}$ and
for $Q_2 = \Z_v \setminus Q_1$ we have that $\abs{Q_2} = \frac{v}{2}$. For each $b,c \in Q_2$, let
$\{ a,x,b,c\}$ be the unique quadruple which contain $\{x,b,c\}$. Since $v \equiv 2 ~ \text{or} ~ 10 ~(\mmod ~12)$,
we can write $\frac{v}{2}=2k+1$.
There are $\binom{2k+1}{2}$ pairs in $Q_2$ and at most $k$ of them
can be in quadruple of the form $\{\alpha,x,b,c\}$, where $\alpha$ is a given point in $Q_1$.
Therefore, by enumerating the number of elements from $Q_1$ that can join $x$ for a pair $\{ x,\alpha\}$ and the maximum number
of pairs from $Q_2$ that can be together with $\{x,\alpha\}$, we must have that
$$
2k \geq \binom{2k+1}{2} \big/ k=2k+1 ~,
$$
a contradiction. Thus, in each nested SQS$(v)$, on $\Z_v$,
each element is contained in at least $\frac{v}{2}$ ND-pairs.
\end{proof}
\begin{corollary}
\label{cor:necc_min_NZ_odd}
If $v \equiv 2 ~ \text{or} ~ 10 ~(\mmod ~12)$, then the minimum number of ND-pairs in a nested
SQS$(v)$ is at least $\frac{v^2}{4}$.
\end{corollary}

The next lemma provides an upper bound on the multiplicity of the ND-pairs with minimum multiplicity.
\begin{lemma}
\label{lem:upper_multi_min}
At least one ND-pair in a nested SQS$(v)$ has multiplicity at most $\frac{v-1}{3}$.
\end{lemma}
\begin{proof}
Let $(Q,\cB)$ be a nested SQS$(v)$. By Lemma~\ref{lem:total_pairs}, the total number of pairs (with repetitions) is
$\frac{v(v-1)(v-2)}{12}$. By Corollary~\ref{cor:lower_nonzeroPN},
the number of ND-pairs in a nested SQS$(v)$ is
at least $\frac{v}{2} ( \frac{v}{2} -1 )$. Hence, the multiplicity of one ND-pair in the nested SQS is at most
$\frac{v(v-1)(v-2)}{12} \big/ \frac{v}{2} ( \frac{v}{2} -1 ) = \frac{v-1}{3}$.
\end{proof}

The next lemma provides a lower bound on the multiplicity of the ND-pairs with maximum multiplicity.
\begin{lemma}
\label{lem:upper_multi_max}
At least one ND-pair in a nested SQS$(v)$ has multiplicity at least $\frac{v-2}{6}$.
\end{lemma}
\begin{proof}
Let $(Q,\cB)$ be a nested SQS$(v)$. By Lemma~\ref{lem:total_pairs}, the total number of pairs (with repetitions) is
$\frac{v(v-1)(v-2)}{12}$. Clearly, the maximum number of ND-pairs is $\binom{v}{2}$.
Hence, the multiplicity of one ND-pair in the nested SQS$(v)$ is
at least ${\frac{v(v-1)(v-2)}{12} \big/ \binom{v}{2}  = \frac{v-2}{6}}$.
\end{proof}

The bounds that were obtained in this section are summarized in Section~\ref{sec:conclude}.
It is also summarized which construction in Sections~\ref{sec:recursive} and~\ref{sec:uniform} attain these bounds.

\section{Recursive Constructions}
\label{sec:recursive}

There are many recursive constructions for SQS$(v)$. In our exposition several such constructions will be used and
some of them will be presented now. One important ingredient in some constructions is a one-factorization of
the complete graph $K_v$, where $v$, the number of vertices, is an even integer. A~{\bf \emph{one-factor}} in $K_v$ is
a perfect matching, i.e., a set of $v/2$ pairs of vertices (i.e., edges), such that two distinct pairs do not have a vertex in common.
A~{\bf \emph{one-factorization}} of the complete graph $K_v$ is a partition of the edges of $K_v$ into $v-1$ one-factors.

The two constructions that will be presented are for nested SQS$(2v)$ based on nested SQS$(v)$.
This type of construction is known as a doubling construction. There are a few such constructions in the
literature and they are used for example also for other structures such as pairwise disjoint quadruple
systems~\cite{Car56,Han60,LiRo78,Wit38}.
For the two representative constructions which will be presented for the SQS$(2v)$
we assume that the existence of a nested SQS$(v)$. The first doubing construction was used for example in
constructions of pairwise disjoint SQS~\cite{Etz93,EtHa91,EtZh20,Lin77,Lin85}

\begin{construction}
\label{cons:2vA}
Let $\cS=(Q,\cB)$ be a nested SQS$(v)$, where $Q$ is the set of $v$ points and $\cB$ is the set of blocks.
Let $\cF = \{ F_1,F_2,\ldots,F_{v-1} \}$ be a one-factorization of~$K_v$ on the points of $Q$.
Construct the set $\cD$ of blocks on the set of points $Q \times \{0,1\}$ as follows:

\vspace{0.1cm}

\noindent
{\bf Type I:}
$$
\{ (x,i),(y,i)),(z,i),(w,i) \} , ~~ \textup{for} ~~ \{x,y,z,w\} \in \cB , ~~ i \in \{0,1\}.
$$

\noindent
{\bf Type II:}
$$
\{ (x,0),(y,0),(z,1),(w,1) \} , ~~ \textup{for} ~~ \{ x,y \} \in F_i, ~~ \{ z,w \} \in F_i, ~~  1 \leq i \leq v-1 ~.
$$

The partitions into pairs are as follows:

For Type I, a block $\{ (x,i),(y,i)),(z,i),(w,i) \}$ is partitioned into the pairs $\{ (x,i),(y,i) \}$ and $\{ (z,i),(w,i) \}$ if and only if
the block $\{ x,y,z,w \}$ of $\cB$ was partitioned into $\{ x,y \}$ and $\{ z,w \}$.

For Type II, a block $\{ (x,0),(y,0),(z,1),(w,1) \}$ is partitioned into the pairs $\{ (x,0),(y,0)\}$ and $\{(z,1),(w,1) \}$.
\end{construction}

In the following two lemmas the number of ND-pairs of Construction~\ref{cons:2vA} and their multiplicity is analyzed.

\begin{lemma}
\label{lem:nonzeroP_ConA}
The number of ND-pairs that are contained in the nested SQS$(2v)$ obtained
by Construction~\ref{cons:2vA} is $v(v-1)$. Each element of $Q \times \{0,1\}$ is contained
in exactly $v-1$ ND-pairs.
\end{lemma}
\begin{proof}
The ND-pairs are exactly all the elements of the form $\{ (a,i),(b,i) \}$, where $a,b \in Q$ and $i \in \Z_2$ and hence
each element of $Q \times \{0,1\}$ is contained in exactly $v-1$ ND-pairs.
For each $i$ there are $\binom{v}{2}$ such pairs which implies total number of ND-pairs.
\end{proof}

\begin{corollary}
\label{cor:nonzeroP_ConA}
The lower bounds of Lemma~\ref{lem:lower_nonzeroPN} and Corollary~\ref{cor:lower_nonzeroPN}
on the number ND-pairs that contain each element and the number of ND-pairs is attained by Construction~\ref{cons:2vA}
\end{corollary}

\begin{lemma}
\label{lem:multi_ConA}
The multiplicity of $\{(x,0),(y,0)\}$ and $\{(x,1),(y,1)\}$
in the nested SQS$(2v)$ of Construction~\ref{cons:2vA} is $\frac{v}{2} +\mu$ if and only if in $\cS$
the pair $\{x,y\}$ has multiplicity $\mu$.
\end{lemma}
\begin{proof}
For each pair $\{x,y\}$ and $i \in \Z_2$, where $x,y \in Q$, the pair $\{ (x,i),(y,i) \}$ is contained $\frac{v}{2}$~times
in the partition of the blocks of Type II into pairs. If the pair $\{ x,y \}$ has multiplicity~$\mu$ in $\cS$,
then the pair $\{ (x,0),(y,0) \}$ and the pair $\{ (x,1),(y,1) \}$ is contained $\mu$ times
in the partition of Type I of the blocks in $\cD$.
Therefore, the multiplicity of $\{ (x,0),(y,0) \}$ and $\{ (x,1),(y,1) \}$ is $\frac{v}{2} +\mu$ in the partition of the blocks in $\cD$.
\end{proof}

For the doubling construction which was suggested in Construction~\ref{cons:2vA} we have that in the SQS$(2v)$ each ND-pair
has multiplicity at least $\frac{v}{2}$. By Lemma~\ref{lem:multi_ConA} there are ND-pairs with multiplicity
$\frac{v}{2}$ in $\cD$ if and only if not all the pairs of $Q$ are ND-pairs in $\cS$.

The second construction was presented by Hanani~\cite{Han60}.
In this celebrated paper he mentioned that doubling constructions were presented and proved before in~\cite{Car56,Wit38}.
Generally, each quadruple can be partitioned into pairs in three distinct ways.

\begin{construction}
\label{cons:2vB}
Let $\cS=(Q,\cB)$ be a nested SQS$(v)$ having $\binom{v}{2}$ ND-pairs.
Construct the following set of blocks $\cD$ on the point set $Q \times \{ 0,1 \}$.

\vspace{0.1cm}

\noindent
{\bf Type I:}
$$
\{ (x,i),(y,j),(z,k),(w,m) \}, ~~ \textup{for} ~~ \{ x,y,z,w \} \in \cB, ~~ i,j,k,m \in \Z_2,~~ i+j+k+m \equiv 0 ~(\mmod ~ 2)
$$

\noindent
{\bf Type II:}
$$
\{ (x,0),(x,1),(y,0),(y,1) \}, ~~ \textup{for} ~~  x,y \in Q, ~~ x \neq y ~.
$$

Assume that in $\cS$ the block $\{ x,y,z,w \}$ was partitioned
into the pairs $\{ x,y \}$ and $\{ z,w \}$. Then, in $\cD$ the block $\{ (x,i),(y,j),(z,k),(w,m) \}$
of Type I will be partitioned into the pairs $\{ (x,i),(y,j) \}$ and $\{ (z,k),(w,m) \}$.

For a block of Type II, $\{ (x,0),(x,1),(y,0),(y,1) \}$, the partition will be into the pairs $\{ (x,0),(x,1) \}$ and $\{ (y,0),(y,1) \}$.
\end{construction}

By considering the partition of each block of Type II in Construction~\ref {cons:2vB} we have the following immediate observation.
\begin{lemma}
\label{lem:all_v}
In Construction~\ref{cons:2vB} each one of the $v$ pairs $\{ (a,0),(a,1) \}$, where $a \in Q$, is an ND-pair.
\end{lemma}

Now, we would like to know the multiplicity of the ND-pairs in the nested SQS done in Construction~\ref{cons:2vB}.
\begin{lemma}
\label{lem:multi_ConB}
The multiplicity of the ND-pair $\{(a,i),(b,j)\}$, where $a,b \in Q$, $a \neq b$, from the blocks of Type I in
Construction~\ref{cons:2vB} is $2\mu$, where $\mu$ is the multiplicity of the associated ND-pair $\{a,b\}$ in $\cS$.
The multiplicity of each of the $v$ ND-pair $\{ (a,0),(a,1) \}$, where $a \in Q$, is $v-1$.
\end{lemma}
\begin{proof}
If a block $\{ a,b,c,d \}$ in $\cB$ is partitioned into the pairs $\{ a,b \}$ and $\{ c,d \}$, then
the associated blocks in $\cD$ is $\{ (a,i),(b,j),(c,k),(d,m) \}$ and for fixed $i$ and $j$ there are exactly
two assignments for $k$ and $m$ such that $i+j+k+m \equiv 0 ~ (\mmod ~ 2)$.
Therefore, if the ND-pair $\{ a,b \}$ has multiplicity $\mu$ in $\cS$,
then the ND-pair $\{ (a,i),(b,j) \}$, $i,j \in \{ 0,1 \}$, has multiplicity $2 \mu$ in $\cD$.

For each $x \in Q$, each block of Type II, $\{ \{ (x,0),(x,1),(y,0),(y,1)\} ~:~ y \in Q \setminus \{x\} \}$
is partitioned into the pairs $\{ (x,0),(x,1) \}$ and $\{ (y,0),(y,1) \}$ and hence $\{ (x,0),(x,1) \}$
is an ND-pair. Since $\abs{\{ y : y \in Q \setminus \{x\} \}}=v-1$, it follows that its multiplicity is $v-1$.
\end{proof}

\begin{corollary}
\label{cor:all_2v}
In Construction~\ref{cons:2vB}, all the $\binom{2v}{2}$ pairs of $Q \times \{ 0,1 \}$ are ND-pairs.
\end{corollary}
\begin{proof}
In $\cS$ all the $\binom{v}{2}$ pairs of $Q$ are ND-pairs and hence from Type I of Construction~\ref{cons:2vB} by Lemma~\ref{lem:multi_ConB}
we have $4\binom{2v}{2}$ ND-pairs, By Lemma~\ref{lem:all_v} we have additional $v$ ND-pairs of Type II and in total
$4\binom{2v}{2} +v=\binom{2v}{2}$ ND-pairs in $\cD$ obtained in Construction~\ref{cons:2vB}.
\end{proof}

\begin{corollary}
\label{cor:attain_upper}
The upper bounds on the multiplicity of ND-pairs and the number of ND-pairs with this multiplicity
that were proved in Lemmas~\ref{lem:max_multiplicity} and~\ref{lem:maxPairsMax} are attained by Construction~\ref {cons:2vB}.
\end{corollary}

\begin{remark}
\label{rem:diff_pairs}
The elements of a quadruple can be partitioned into pairs in three different ways.
For example, in Construction~\ref{cons:2vB} each quadruples of Type II can also be partitioned in on of the following ways:
\begin{enumerate}
\item $\{ (x,0),(y,1) \}$ and $\{ (y,0),(x,1) \}$;
\item $\{ (x,0),(y,0) \}$ and $\{ (x,1),(y,1) \}$.
\end{enumerate}
Of course each quadruple can be partitioned in a different way. A different partition implies different number of ND-pairs and
different multiplicities for these pairs. Bounds which are met by one partition of the quadruples might not be
met with the other partition.
\end{remark}

Starting with a construction of a nested SQS$(v)$ in which all the pairs are ND-pairs and applying Construction~\ref{cons:2vB}
yields a nested SQS$(2v)$ in which all the pairs are also ND-pairs. To have a nested SQS in which the ND-pairs
have larger multiplicity we can apply Construction~\ref{cons:2vA}.

We start with Construction~\ref{cons:2vA} from an SQS$(v)$, $(Q,\cB)$.
By Lemma~\ref{lem:multi_ConA} the generated SQS$(2v)$ has ND-pairs
with minimum multiplicity~$\frac{v}{2}$ if there are pairs of $Q$ which are not ND-pairs.
If all the pairs of $Q$ are ND-pairs in the SQS$(v)$ and the
smallest multiplicity is $\mu$, then by Lemma~\ref{lem:multi_ConA} in the SQS$(2v)$ obtained by the construction
the smallest multiplicity of an ND-pair is $\frac{v}{2} + \mu$.

We continue with Construction~\ref{cons:2vB} from an SQS$(v)$, $(Q,\cB)$. In the constructed SQS$(2v)$
all pairs are ND-pairs since in the SQS$(v)$ of the construction all the pairs are ND-pairs. If
the ND-pair with the smallest multiplicity has multiplicity $\mu$, then in the SQS$(2v)$ the ND-pair with the smallest multiplicity
has multiplicity~$2 \mu$. By Lemma~\ref{lem:upper_multi_min} the smallest multiplicity of a ND-pair
cannot be more than $\frac{v-1}{3}$. If such multiplicity is attained it implies
multiplicity $\frac{2(v-1)}{3}$ for the ND-pairs in the
nested SQS$(2v)$, while by Lemma~\ref{lem:multi_ConB} there will be some ND-pairs with larger multiplicity $v-1$.
By Lemma~\ref{lem:max_multiplicity}, this is the maximum multiplicity that can be obtained for any ND-pair.

\section{Blocks Partition with Uniform Multiplicity}
\label{sec:uniform}

In this section we will be interested in nested SQS$(v)$, $(Q,\cB)$, in which all the ND-pairs
have the same multiplicity. Such a nested SQS will be called a {\bf \emph{uniform nested SQS}}.
In particular we will be interested in such uniform nested SQSs where
all the $\binom{v}{2}$ pairs of $Q$ are ND-pairs
and also in such uniform nested SQSs when the number of ND-pairs attains the lower bound
of Corollary~\ref{cor:lower_nonzeroPN}, i.e., $\frac{v}{2}(\frac{v}{2}-1)$ (which is the minimum number of ND-pairs).
The first result indicates when uniform nested SQS can exist.

\begin{theorem}
\label{thm:all_uniform}
In a uniform nested SQS$(v)$,
if all the $\binom{v}{2}$ pairs in the nested SQS$(v)$ are ND-pairs, then $v \equiv 2 ~ (\mmod ~ 6)$ and the
multiplicity of each pair is $\frac{v-2}{6}$.
\end{theorem}
\begin{proof}
By Lemma~\ref{lem:total_pairs}, the total number of pairs in the blocks of a nested SQS$(v)$
is $\binom{v}{3} \big/ 2$. The number of pairs in a set of size $v$ is $\binom{v}{2}$ and hence if all of them
are ND-pairs in the nested SQS with the same multiplicity, then the multiplicity of each ND-pair is
$$
\frac{\binom{v}{3}}{2\binom{v}{2}}= \frac{v-2}{6} .
$$
Since SQS$(v)$ exists if and only if $v \equiv 2 ~\text{or} ~ 4 ~ (\mmod ~ 6)$, it follows that for $(v-2)/6$ to be an integer
we must have $v \equiv 2 ~ (\mmod ~ 6)$.
\end{proof}

A nested SQS$(v)$ which satisfies the properties of Theorem~\ref{thm:all_uniform} will be
called a {\bf \emph{complete uniform nested SQS}}. Such a nested SQS also attains the lower bound on
the minimum multiplicity of the ND-pair with the largest multiplicity (see Lemma~\ref{lem:upper_multi_max}).

\begin{example}
\label{ex:SQS8uniform}
By Theorem~\ref{thm:all_uniform} if $v \equiv 2 ~ (\mmod ~ 6)$, then a complete uniform nested SQS$(v)$, $(Q,\cB)$, in which each
of the $\binom{v}{2}$ pairs of $Q$ is a ND-pair, can exist.
Such a complete uniform nested SQS$(8)$ is defined as follows:

$$
[ \{\infty,0 \},\{2,6\}], ~ [ \{\infty,1 \},\{0,3\}], ~ [ \{\infty,2 \},\{1,4\}], ~[ \{\infty,3 \},\{2,5\}],
$$
$$
 ~[ \{\infty,4 \},\{3,6\}], ~[ \{\infty,5 \},\{0,4\}], ~ [ \{\infty,6 \},\{1,5\}], ~ [ \{0,1 \},\{4,6\}],
$$
$$
[ \{0,2 \},\{3,4\}], ~ [ \{0,5 \},\{1,2\}], ~ [ \{0,6 \},\{3,5\}], ~[ \{1,3 \},\{4,5\}],
$$
$$
[ \{1,6 \},\{2,3\}], ~[ \{2,4 \},\{5,6\}].
$$

\hfill\quad $\blacksquare $
\end{example}

In Theorem~\ref{thm:all_uniform} the requirement is that all $\binom{v}{2}$ pairs of $Q$ are ND-pairs and in
this case by the theorem a nested SQS with the same multiplicity for all the ND-pairs implies that $v \equiv 2 ~ (\mmod ~ 6)$.
If the requirement is relaxed in a way that not all the $\binom{v}{2}$ pairs of $Q$ are ND-pairs,
but the ND-pairs have the same multiplicity. In this case, there
are pairs of $Q$ which are not ND-pairs and we can have a similar result for $v \equiv 4 ~ (\mmod ~ 6)$.
In fact, for each $v \equiv 2 ~ \text{or} ~ 4 ~ (\mmod ~ 6)$, there can be a few sets of parameters (number of ND-pairs
and their multiplicity) which yield a uniform nested SQS.
Assume now, that there are $\frac{v}{2} \left( \frac{v}{2} -1  \right)$ ND-pairs. This is the case for example in
Construction~\ref{cons:2vA} (see lemma~\ref{lem:nonzeroP_ConA}).
\begin{theorem}
\label{thm:part_uniform}
In a uniform nested SQS$(v)$,
if there are exactly $\frac{v}{2} \left( \frac{v}{2} -1  \right)$ ND-pairs in the nested SQS$(v)$,
then $v \equiv 4 ~ (\mmod ~ 6)$ and the multiplicity of each ND-pair is $\frac{v-1}{3}$.
\end{theorem}
\begin{proof}
By Lemma~\ref{lem:total_pairs}, the total number of pairs in the nested SQS$(v)$ is $v(v-1)(v-2)/12$ and when
divided by $\frac{v}{2} \left( \frac{v}{2} -1  \right)$
the result is $\frac{v-1}{3}$. This implies that $v \equiv 4 ~ (\mmod ~ 6)$.
\end{proof}

A nested SQS which satisfies the properties of Theorem~\ref{thm:part_uniform} attains the upper bound of
Lemma~\ref{lem:upper_multi_min}, i.e., it has the maximum multiplicity for the ND-pairs with the smallest multiplicity.

By Corollary~\ref{cor:lower_nonzeroPN} and Lemma~\ref{lem:upper_multi_min} a nested SQS$(v)$ which satisfies the properties
of Theorem~\ref{thm:part_uniform} has the minimum number of ND-pairs and the largest multiplicity of an ND-pair
with the smallest multiplicity. A nested SQS$(v)$ which satisfies the properties of Theorem~\ref{thm:part_uniform} will be
called a {\bf \emph{minimum uniform nested SQS}}.

\begin{theorem}
\label{thm:mi_uniform_construct}
If there exists a complete uniform nested SQS$(v)$, then there exists a minimum uniform nested SQS$(2v)$.
\end{theorem}
\begin{proof}
Apply Construction~\ref{cons:2vA} on a complete uniform nested SQS$(v)$, $(Q,\cB)$. By Lemma~\ref{lem:nonzeroP_ConA},
the number of ND-pairs of the SQS$(2v)$ obtained by the construction is $v(v-1)$.
By Lemma~\ref{lem:multi_ConA} the multiplicity of each such pair is $\frac{v}{2} + \mu$, where $\mu = \frac{v-2}{6}$
since the SQS$(v)$ is a complete uniform nested SQS. Hence, the multiplicity of each ND-pair for the nested SQS$(2v)$
is $\frac{v}{2} + \frac{v-2}{6} = \frac{2v-1}{3}$ as required.
\end{proof}

The necessary conditions of Theorem~\ref{thm:part_uniform} are not sufficient for the existence of a minimum uniform nested SQS.
This is for example a consequence of Corollary~\ref{cor:necc_min_NZ_odd}.

\begin{corollary}
\label{cor:No_min_2_10_mod12}
If $v \equiv 10 ~(\mmod ~12)$, then there is no minimum uniform nested SQS$(v)$, i.e., if $\cS$ is
minimum uniform nested SQS$(v)$, then $v \equiv 4 ~(\mmod ~12)$.
\end{corollary}

We continue with some specific examples of complete uniform nested SQS$(v)$s for small values of~$v$.

\begin{example}
\label{ex:ro20}
{\bf \underline{A complete uniform nested rotational SQS$(20)$}:}

Each 19 blocks have one base block from which a complete uniform nested SQS$(20)$ is constructed.
The following 15 base blocks were found by computer search done by Daniella Bar-Lev:

$$
[ \{ \infty , 1 \},\{ 0,8 \} ], ~[ \{ \infty , 2 \},\{ 0,5 \} ], ~[ \{ \infty , 13 \} , \{ 0,9 \} ],~[  \{ 0, 1 \},\{ 2,4 \} ],~[\{0,1\},\{6,9\}],
$$
$$
[\{0,1\},\{10,17\}],~[\{0,2\},\{6,14\}],~[\{0,2\},\{9,15\}],~[\{0,3\},\{4,16\}],~[\{0,3\},\{5,10\}],
$$
$$
[\{0,4\},\{5,9\}],~[\{0,4\},\{7,15\}],~[\{0,5\},\{6,16\}],~[\{0,6\},\{11,18\}],~[\{0,6\},\{8,17\}].
$$
The multiplicity of each ND-pair in this example of a complete uniform nested SQS$(20)$ is three.

\hfill\quad $\blacksquare $
\end{example}

\begin{example}
\label{ex:ro26}
{\bf \underline{A complete uniform nested rotational SQS$(26)$}:}

Each 25 blocks have one base block from which a complete uniform nested SQS$(26)$ is constructed.
The 26 base blocks are taken from~\cite{JiZh02} and the partition of each block into pairs is as follows:

$$
[ \{ \infty , 3 \},\{ 0,1 \} ], ~[ \{ \infty , 13 \},\{ 0,5 \} ], ~[ \{ \infty , 14 \} , \{ 0,7 \} ],~[  \{ \infty , 15 \},\{ 0,6 \} ],~[\{0,1\},\{12,22\}],
$$
$$
[\{0,1\},\{13,21\}],~[\{0,1\},\{14,23\}],~[\{0,2\},\{1,5\}],~[\{0,2\},\{6,9\}],~[\{0,2\},\{7,17\}],
$$
$$
[\{0,2\},\{15,20\}],~[\{0,3\},\{5,11\}],~[\{0,3\},\{8,17\}],~[\{0,3\},\{13,20\}],~[\{0,4\},\{2,12\}],
$$
$$
[\{0,4\},\{6,18\}],~[\{0,4\},\{10,21\}],~[\{0,5\},\{6,24\}],~[\{0,5\},\{9,21\}],~[\{0,6\},\{3,10\}],
$$
$$
[\{0,6\},\{11,22\}],~[\{0,8\},\{ 1,9\}],~[\{0,8\},\{6,19\}],~[\{0,9\},\{7,18\}],~[\{0,9\},\{10,24\}],
$$
$$
[\{0,10\},\{7,19\}].
$$
The multiplicity of each ND-pair in this example of a complete uniform nested SQS$(26)$ is four.

\hfill\quad $\blacksquare $
\end{example}

\begin{example}
\label{ex:ro38}
{\bf \underline{A complete uniform nested rotational SQS$(38)$}:}

Each 37 blocks have one base block from which a complete uniform nested SQS$(38)$ is constructed.
The 57 base blocks are taken from~\cite{JiZh02} and the partition of each block into pairs is as follows:

$$
[\{\infty,22 \}, \{ 0,2 \}],~[ \{\infty,33 \},\{ 0,3 \} ],~[\{ \infty,14 \},\{ 0,8 \}],~[ \{\infty,10 \},\{0,11\} ], ~[ \{\infty,18\},\{ 0,13 \}],
$$
$$
[ \{\infty,28 \},\{0,16\} ],~ [ \{ 0,1 \} , \{ 2,6 \}],~ [ \{ 0,1 \} , \{ 7,19 \}],~[ \{ 0,1 \} , \{ 9,35 \}],~[ \{ 0,1 \} , \{ 10,24 \}],
$$
$$
[ \{ 0,1 \} , \{ 13,30 \}],~[ \{ 0,1 \} , \{ 20,34 \}],~[ \{ 0,2 \} , \{ 4,16 \}],~[ \{ 0,2 \} , \{ 5,8 \}],~[ \{ 0,2 \},\{ 13,36 \} ],
$$
$$
[ \{ 0,2 \} , \{ 19,24 \}],~ [ \{ 0,2 \} , \{ 25,30 \}], ~ [ \{ 0,3 \} , \{ 7,28 \}], ~[ \{ 0,3 \} , \{ 9,19 \}],~[ \{ 0,3 \} , \{ 12,20 \}],
$$
$$
[ \{ 0,3 \} , \{17,27 \}],~ [ \{ 0,4 \} , \{ 1,11 \}],~ [ \{ 0,4 \} , \{ 5,21 \}],~[ \{ 0,4 \} , \{ 8,22 \}],~[ \{ 0,4 \} , \{ 15,27 \}],
$$
$$
[ \{ 0,4 \} , \{ 17,32 \}],~[ \{ 0,5 \} , \{ 3,26 \}],~[ \{0,5 \} , \{ 19,35 \}],~ [ \{ 0,5 \} , \{23,29 \}],~[ \{ 0,5 \} ,\{27,34 \}],
$$
$$
[ \{ 0,6 \} , \{ 4,30 \}],~[ \{ 0,6 \} , \{ 5,31\}] ,~[ \{ 0,6 \} , \{ 7,27 \}],~[ \{ 0,6 \} , \{ 10,32 \}],~ [ \{ 0,6 \} , \{ 17,26 \}],
$$
$$
[ \{ 0,7 \} , \{ 5,25 \}],~[ \{ 0,7 \} , \{ 12,23 \}],~[ \{ 0,7 \} , \{ 14,31 \}],~[ \{ 0,7 \} , \{ 17,36 \}],~[ \{ 0,7 \} , \{ 26,35 \}],
$$
$$
[ \{ 0,8 \} , \{ 6,23 \}],~~[ \{ 8,0 \}, \{ 9,33 \}],~[ \{ 0,8 \} , \{ 10,35 \}],~[ \{ 0,8 \} , \{ 11,24 \}], ~[ \{ 0,9 \} , \{ 4,14 \}],
$$
$$
[ \{ 0,9 \} , \{ 6,24 \}],~[ \{ 0,9 \} , \{ 10,21 \}],~[ \{ 0,9 \} , \{ 15,30 \}],~[ \{ 0,10 \} , \{ 4,19 \}], ~[ \{ 0,10 \} , \{ 22,34 \}],
$$
$$
[ \{ 0,12 \} , \{ 8,32 \}], ~[ \{ 0,13 \} , \{ 9,29 \}],~[ \{ 0,13 \} , \{ 16,35 \}],~[ \{ 0,14 \} , \{ 15,36 \}],~[ \{ 0,15 \} , \{ 3,21 \}],
$$
$$
[ \{ 0,15 \} , \{ 8,26 \}],~[ \{ 0,16 \}, \{ 8,27 \}].
$$
The multiplicity of each ND-pair in this example of a complete uniform nested SQS$(38)$ is six.

\hfill\quad $\blacksquare $
\end{example}

\begin{example}
\label{ex:ro62}
{\bf \underline{A complete uniform nested rotational SQS$(62)$}:}

This rotational SQS$(62)$ has a multiplier group $\{1,9,20,34,58\}$ of $\Z_{61}$.
Each 305 blocks have one base block from which a complete uniform nested SQS$(62)$ is constructed.
The 31 base blocks are taken from~\cite{JiZh02} and the partition of each block into pairs is as follows:

$$
[\{\infty,5 \}, \{ 0,1 \}],~[ \{\infty,10 \},\{ 0,2 \} ],~[ \{0,1 \},\{20,30\} ],~[\{0,1 \},\{ 21,42 \}], ~[ \{0,1\},\{ 22,23 \}],
$$
$$
[ \{0,1 \},\{24,25\} ],~ [ \{ 0,1 \} , \{ 26,31 \}],~ [ \{ 0,1 \} , \{ 29,41 \}],~[ \{ 0,1 \} , \{ 33,57 \}],~[ \{ 0,2 \} , \{ 14,44 \}],
$$
$$
[ \{ 0,2 \} , \{ 18,56 \}],~[ \{ 0,2 \} , \{ 20,22 \}],~[ \{ 0,2 \} , \{ 30,53 \}],~[ \{ 0,2 \},\{ 33,49 \} ],~[ \{ 0,2 \} , \{ 40,42 \}],
$$
$$
[ \{ 0,4 \} , \{ 7,58 \}], ~[ \{ 0,4 \} , \{ 8,53 \}],~[ \{ 0,4 \} , \{ 9,13 \}],~ [ \{ 0,4 \} , \{ 12,28 \}], ~[ \{ 0,4 \} , \{ 17,43 \}],
$$
$$
[ \{ 0,4 \} , \{22,46 \}],~[ \{ 0,4 \} , \{ 33,55 \}],~  [ \{ 0,4 \} , \{ 36,44 \}],~ [ \{ 0,5 \} , \{ 14,52 \}],~[ \{0,5 \} , \{ 18,28 \}],
$$
$$
[ \{ 0,5 \} , \{ 34,39 \}],~[ \{ 0,5 \} , \{ 48,56 \}],~ [ \{ 0,8 \} , \{1,9 \}],~[ \{ 0,8 \} , \{ 21,51 \}],~[ \{ 0,10 \} ,\{1,11 \}],
$$
$$
[ \{ 0,10 \} , \{ 12,56 \}].
$$
The multiplicity of each ND-pair in this example of a complete uniform nested SQS$(62)$ is ten.

\hfill\quad $\blacksquare $
\end{example}

A construction for an infinite family of uniform nested SQSs is one of the basic goals for this research.
The most attractive target is to construct a uniform nested SQS$(2^n)$. A~complete uniform nested
SQS$(8)$ was given in Example~\ref{ex:SQS8uniform}. This implies by Theorem~\ref{thm:mi_uniform_construct}
a minimum uniform nested SQS$(16)$. The next one is SQS$(32)$. A recursive construction is based on the blocks
of a minimum uniform nested SQS$(16)$ and some more structures of block design. The construction is quite complicated and since we have not
managed to generalize it, we omit its description.
We will describe now another construction that yields a complete uniform nested SQS$(32)$.
The description will be general for SQS$(2^n)$, but its implementation will be only for SQS$(32)$.

\begin{construction}
\label{cons:Bool}
Let $\cS=(\Z_{2^n-1} \cup \{ \infty \},\cB)$ be a Boolean SQS$(2^n)$, where $n$ is odd integer and the elements of $\Z_2^n \setminus \{0\}$ are
identified by the powers of a primitive element $\alpha$ in GF$(2^n)$.
The block $\{ \infty, i,j,k \}$ is in $\cS$
if and only if $\alpha^i + \alpha^j + \alpha^k = 0$ and the block $\{ i,j,k,m \}$ is in $\cS$
if and only if $\alpha^i + \alpha^j + \alpha^k + \alpha^m = 0$.
Two blocks whose elements differ only by a multiplication of some power of two will be considered to be
in the same {\bf \emph{block class}}.
For example, if the block $\{ \infty ,x,y,z \}$ is in~$\cB$, then the block $\{ \infty, 2x,2y,2z \}$ is also in $\cB$.
Also, if the block $\{ \infty ,x,y,z \}$ is in~$\cB$, then
the block $\{ \infty, 2^n - 1 -x,2^n -1 -y,2^n -1 -z \}$ is also in $\cB$ and in the same block class.
Similarly, if the block $\{ x,y,z,w \}$ is in~$\cB$, then the block $\{ 2x,2y,2z,2w \}$ is also in $\cB$
and the block $\{ 2^n-1-x,2^n-1-y,2^n-1-z,2^n-1-w \}$ is also in $\cB$.

For each block class one base block is chosen. The partition of this base block into pairs implies the same
partition into pairs for each block in its block class.
For example, if $\{ x,y,z,w \}$ is partitioned into $\{ x,y \}$ and $\{z,w\}$,
then $\{ 2x,2y,2z,2w \}$ is partitioned into $\{ 2x,2y \}$ and $\{2z,2w\}$,
and $\{ 2^n-1-x,2^n-1-y,2^n-1-z,2^n-1-w \}$ is partitioned into $\{ 2^n-1-x,2^n-1-y \}$ and $\{2^n-1-z,2^n-1-w\}$.
\end{construction}

\begin{example}
\label{ex:Bool32}
We apply all the given rules of Construction~\ref{cons:Bool} to find the partition
of each block for the Boolean SQS$(32)$ represented as a rotational SQS.
Let $\alpha$ be a root of the primitive polynomial ${x^5 + x^2 +1}$, i.e., $\alpha$ is a primitive element
in GF$(32)$. The points of the system are $Z_{31} \cup \{\infty \}$. The block $\{ \infty, i,j,k \}$ is in the system
if and only if $\alpha^i + \alpha^j + \alpha^k = 0$ and the block $\{ i,j,k,m \}$ is in the system
if and only if $\alpha^i + \alpha^j + \alpha^k + \alpha^m = 0$.
The SQS$(32)$ has exactly eight block classes and here is such a partition for each block to obtain
a complete uniform nested SQS$(32)$ from the eight base blocks of the block classes for this SQS$(32)$:
$$
B_1=[ \{ 0,1 \} , \{ 2,11 \}], ~B_2=[\{0,1\},\{3,27\}],~B_3=[\{0,1\},\{4,17\}],~B_4=[\{0,1\},\{5,19\}],
$$
$$
B_5=[\{0,26\},\{1,7\}],~B_6=[\{0,24\},\{1,14\}],~B_7=[\{0,5\},\{10,22\}],~B_8=[\{0,2\},\{\infty,5\}].
$$

\hfill\quad $\blacksquare $
\end{example}

\begin{lemma}
\label{lem:SQS32}
Example~\ref{ex:Bool32} yields a complete uniform nested SQS$(32)$.
\end{lemma}
\begin{proof}
Clearly by the definition of the blocks in the system a Boolean SQS$(32)$ was defined.
Consider the six cyclotomic cosets of size 5 modulo 31, i.e.,
The set of integers modulo $31$, $\{ x, 2x, \ldots, 16 x \}$, $x \notin \{0,31 \}$, is a {\bf \emph{cyclotomic coset}} modulo $31$.
$$
C_1 = \{ 1,2,4,8,16 \}, ~~~ C_{15} = \{ 30,29,27,23,15 \},
$$
$$
C_3 = \{ 3,6,12,24,17 \}, ~~~ C_7 = \{ 28,25,19,7,14\},
$$
$$
C_5 = \{ 5, 10,20,9,18 \}, ~~~ C_{11} = \{ 26,21,11,22,13\}.
$$

The differences in the pairs $\{0,1\}$ and $\{0,2\}$ are elements in $C_1 \cup C_{15}$ and they are contained 5 times,
in $B_1$, $B_2$, $B_3$, $B_4$, and $B_8$. No other difference from $C_1 \cup C_{15}$ is contained in an ND-pair
from the partition of the blocks into pairs. When we generate the block class for each block, say $B_1$, each difference from
$C_1 \cup C_{15}$ will be the difference in exactly one of the blocks from the class.
Since, there are exactly 5 base blocks with differences from $C_1 \cup C_{15}$, it follows that each difference in pairs
from $C_1 \cup C_{15}$ is contained in exactly five base blocks for the rotational SQS$(32)$ which is the required
multiplicity of an ND-pair in a complete uniform nested SQS$(32)$.

The difference 24 is obtained in the ND-pair $\{ 3, 27\}$ which is contained in $B_2$ and is also obtained in the ND-pair
$\{ 0,24 \}$ which is contained in $B_6$. The difference 14 is obtained in the ND-pair $\{ 5,19\}$ which is contained
in $B_4$. The difference 6 is obtained in the ND-pair $\{ 1,7\}$ which is contained in $B_5$.
The difference 12 is obtained in the ND-pair $\{ 10,22\}$ which is contained in $B_7$. These four differences 24, 14, 6, and 12,
are elements in $C_3 \cup C_7$.
Since, there are exactly 5 base blocks with differences from $C_3 \cup C_7$, it follows that each difference in pairs
from $C_3 \cup C_7$ is contained in exactly five base blocks for the rotational SQS$(32)$ which is the required
multiplicity of an ND-pair in a complete uniform nested SQS$(32)$.

The difference 9 is obtained in the ND-pair $\{2,11\}$ which is contained in $B_1$.
The difference 13 is obtained in the ND-pair $\{4,17\}$ which is contained in $B_3$
and is also obtained  in the pair $\{1,14\}$ which is contained in $B_6$.
The difference 26 is obtained in the ND-pair $\{0,26\}$ which is contained in $B_5$.
The difference 5 is obtained in the ND-pair $\{0,5\}$ which is contained in $B_7$.
Since, there are exactly 5 base blocks with differences from $C_5 \cup C_{11}$, it follows that each difference in pairs
from $C_5 \cup C_{11}$ is contained in exactly five base blocks for the rotational SQS$(32)$ which is the required
multiplicity of a ND-pair in a complete uniform nested SQS$(32)$.

The pair $\{ \infty,5\}$ contained in $B_8$ and its multiples by 2, $\{ \infty,10\}$, $\{ \infty,20\}$, $\{ \infty,9\}$, $\{ \infty,18\}$,
together with their 31 cyclic shifts implies that each pair $\{ \infty ,x \}$, $0 \leq x \leq 30$, is a ND-pair with muliplicity 5 as required.
\end{proof}

Example~\ref{ex:Bool32} for a complete uniform nested Boolean SQS$(32)$ was relatively easy to obtain since only
eight base blocks are involved in the construction. Implementation of Construction~\ref{cons:Bool}
for an SQS$(128)$ will be more difficult as 96 base blocks are involved. Larger orders will make it even
extremely harder to obtain a complete uniform nested SQS.

In a minimum uniform nested SQS$(v)$ the number of ND-pairs is $\frac{v}{2} \left( \frac{v}{2} -1  \right)$
and each such pair has the same multiplicity.
But, can we have a uniform nested SQS$(v)$, where the number of ND-pairs is between
$\frac{v}{2} \left( \frac{v}{2} -1  \right)$ and $\binom{v}{2}$. The answer is definitely positive.
There is a unique SQS$(10)$~\cite{Bar08}. This system has 30 blocks (between $\frac{v}{2} \left( \frac{v}{2} -1  \right)=20$
and $\binom{v}{2}=45$). The total number of pairs in a nested SQS$(10)$ is 60.
A uniform nested SQS$(10)$ with 20 ND-pairs, where each one has multiplicity 3 does not exist.
There exists a uniform nested SQS$(10)$ with 30 ND-pairs as presented in the following example.

\begin{example}
\label{ex:SQS10uniform}
$~$
\noindent
Let $\Z_5 \times \Z_2$ be the set of points for an SQS$(10)$ and consider the following set of its blocks
and the partition of each block into pairs:
\noindent
$$
[ \{(0,0),(2,1) \},\{(0,1),(1,1)\}], [ \{(2,0),(3,1) \},\{(1,1),(2,1)\}], [ \{(4,0),(4,1) \},\{(2,1),(3,1)\}],
$$
$$
[ \{(1,0),(0,1) \},\{(3,1),(4,1)\}], ~[ \{(3,0),(1,1) \},\{(4,1),(0,1)\}],
$$
$$
[ \{(0,0),(4,1) \},\{(1,1),(3,1)\}], [ \{(2,0),(0,1) \},\{(2,1),(4,1)\}], [ \{(4,0),(1,1) \},\{(0,1),(3,1)\}],
$$
$$
[ \{(1,0),(2,1) \},\{(1,1),(4,1)\}], ~[ \{(3,0),(3,1) \},\{(0,1),(2,1)\}],
$$
$$
[ \{(1,0),(2,0) \},\{(0,1),(1,1)\}], [ \{(3,0),(4,0) \},\{(1,1),(2,1)\}], [ \{(0,0),(1,0) \},\{(2,1),(3,1)\}],
$$
$$
[ \{(2,0),(3,0) \},\{(3,1),(4,1)\}], ~[ \{(0,0),(4,0) \},\{(0,1),(4,1)\}],
$$
$$
[ \{(1,0),(4,0) \},\{(0,1),(2,1)\}], [ \{(1,0),(3,0) \},\{(1,1),(3,1)\}], [ \{(0,0),(3,0) \},\{(2,1),(4,1)\}],
$$
$$
[ \{(0,0),(2,0) \},\{(0,1),(3,1)\}], ~[ \{(2,0),(4,0) \},\{(1,1),(4,1)\}],
$$
$$
[ \{(1,0),(2,0) \},\{(0,0),(4,1)\}], [ \{(2,0),(3,0) \},\{(1,0),(2,1)\}], [ \{(3,0),(4,0) \},\{(2,0),(0,1)\}],
$$
$$
[ \{(0,0),(4,0) \},\{(3,0),(3,1)\}], ~ [ \{(0,0),(1,0) \},\{(4,0),(1,1)\}], ~
$$
$$
[ \{(1,0),(3,0) \},\{(4,0),(4,1)\}], [ \{(2,0),(4,0) \},\{(0,0),(2,1)\}], [ \{(0,0),(3,0) \},\{(1,0),(0,1)\}],
$$
$$
[ \{(1,0),(4,0) \},\{(2,0),(3,1)\}], ~ [ \{(0,0),(2,0) \},\{(3,0),(1,1)\}]
$$
Each one of the thirty ND-pairs contained in this nested SQS$(10)$ has multiplicity 2 and hence this nested SQS$(10)$ is uniform,
but neither complete nor minimum.

\hfill\quad $\blacksquare $
\end{example}

One necessary condition for the existence of a uniform nested SQS$(v)$ which is neither complete nor minimum
is that the number of ND-pair will be between $\frac{v}{2} (\frac{v}{2} -1)$ and $\binom{v}{2}$ and this number
divides the total number of pairs in the nested SQS$(v)$ which is $\frac{v(v-1)(v-2)}{12}$ by Lemma~\ref{lem:total_pairs}.
But, this necessary condition is not sufficient in many cases. There is another necessary condition for the existence of such
a uniform nested SQS$(v)$.

\begin{lemma}
\label{lem:div_multi_cond}
Let $M$ be an integer between $\frac{v}{2} (\frac{v}{2} -1)$ and $\binom{v}{2}$ and let $\mu = \frac{v(v-1)(v-2)}{12M}$.
A~necessary condition that there exists a uniform nested SQS$(v)$ with $M$ ND-pairs and multiplicity $\mu$
is that $\mu$ divides $\frac{(v-1)(v-2)}{6}$.
\end{lemma}
\begin{proof}
Let $\cS=(Q,\cB)$ be an SQS$(v)$.
Each triple of $\Z_v$ is contained in exactly one block of an SQS$(v)$ and hence if $x \in \Z_v$, then the SQS$(v)$
has $\frac{(v-1)(v-2)}{6}$ blocks which contain $x$. In each such block one ND-pair contains $x$ in the nested SQS$(v)$.
Since each ND-pair is contained exactly $\mu$ times in blocks of the the uniform nested SQS$(v)$, it follows that $\mu$ must divides
the total number of pairs in the nested SQS$(v)$ which contain $x$ and this number is $\frac{(v-1)(v-2)}{6}$.
Therefore, $\mu$ must divides $\frac{(v-1)(v-2)}{6}$.
\end{proof}

\begin{corollary}
\label{cor:div_multi_cond}
Let $M$ be an integer between $\frac{v}{2} (\frac{v}{2} -1)$ and $\binom{v}{2}$ and let $\mu = \frac{v(v-1)(v-2)}{12M}$.
A~necessary condition that there exists a uniform nested SQS$(v)$ with $M$ ND-pairs
is that $v$ divides $2M$.
\end{corollary}
\begin{proof}
%%%%%REVIEWER RECOMMEND TO REMOVE
Let $\ell$ the number of ND-pairs which contain a given point $x$.
Since the nested SQS$(v)$ is uniform, it follows that each point is contained in the same number of ND-pairs.
Hence, the total number of ND-pairs in $\frac{v \ell}{2}$, where the division by 2 is due to the fact that each
ND-pair is counted twice in this enumeration. Therefore, $M=\frac{v \ell}{2}$ and $v$ must divides $2M$.
\end{proof}
Corollary~\ref{cor:div_multi_cond} is presented as a consequence of Lemma~\ref{lem:div_multi_cond} since the conditions
in these two results can be proved to be equivalent.

The following table presents the various values of $v \equiv 2 ~ \text{or} ~ 4 ~ (\mmod ~ 6)$, the minimum number $\frac{v}{2}(\frac{v}{2} -1)$
of ND-pairs in a nested SQS$(v)$, and the maximum number~$\binom{v}{2}$ of ND-pairs. The bold integers
refer to the number of ND-pairs in a possible uniform nested SQS$(v)$  (it must satisfies the
conditions of Corollary~\ref{cor:necc_min_NZ_odd}) and Lemma~\ref{lem:div_multi_cond}. Some of these values are attained
and some are still in question. The other values cannot be attained. In brackets we have
the multiplicity of the associated uniform nested SQS$(v)$ .

$$
\begin{array}{ccccccccccc}
\hline
&& & ~~ && && & ~~~~~~ && \\
v && \text{pairs~in~SQS} & ~~&& \frac{v}{2}(\frac{v}{2} -1) && \binom{v}{2} & ~~~~~~&& \text{nonzero~pairs} \\
&& & ~~ && && & ~~~~~~ && \\
\hline
8 && 28 & ~~ && 12^{\text{x}} && {\bf 28}^{\text{a}}(1) & ~~~~~~ &&   \\
10 && 60 & ~~&& 20^{\text{b}} && 45^{\text{z}} & ~~~~~~ && {\bf 30}^{\text{c}}(2)  \\
14 && 182 & ~~&& 42^{\text{x}} && {\bf 91}^{\text{?}}(2) & ~~~~~~ && \\
16 && 280 & ~~&& {\bf 56}^{\text{d}}(5) && 120^{\text{z}} & ~~~~~~ &&  \\
20 && 570 & ~~&& 90^{\text{x}} && {\bf 190}^{\text{e}}(3) & ~~~~~~ &&  \\
22 && 770 & ~~&& 110^{\text{b}} && 231^{\text{z}} & ~~~~~~ && {\bf 154}^{\text{?}}(5) \\
26 && 1300 & ~~&& 156^{\text{x}} && {\bf 325}^{\text{f}}(4) & ~~~~~~ && {\bf 260}^{\text{?}}(5)  \\
28 && 1638 & ~~&& {\bf 182}(9) && 378^{\text{z}} & ~~~~~~ &&  \\
32 && 2480 & ~~&& 240^{\text{x}} && {\bf 496}^{\text{g}}(5) & ~~~~~~ &&   \\
34 && 2992 & ~~&& 272^{\text{b}} && 561^{\text{z}} & ~~~~~~ && {\bf 374}^{\text{?}}(8)  \\
38 && 4218 & ~~&& 342^{\text{x}} && {\bf 703}^{\text{h}}(6) & ~~~~~~ &&  \\
40 && 4940 & ~~&& {\bf 380}^{\text{d}}(13) && 780^{\text{z}} & ~~~~~~ &&  \\
44 && 6622 & ~~&& 462^{\text{x}} && {\bf 946}^{\text{?}}(7) & ~~~~~~ &&   \\
46 && 7590 & ~~&& 506^{\text{b}} && 1035^{\text{z}} & ~~~~~~ && {\bf 690}^{\text{?}}(11), ~ {\bf 759}^{\text{?}}(10) \\
50 && 9800 & ~~&& 600^{\text{x}} && {\bf 1225}^{\text{?}}(8) & ~~~~~~ && {\bf 700}^{\text{?}}(14)  \\
52 && 11050 & ~~&& {\bf 650}^{\text{d}}(17) && 1326^{\text{z}} & ~~~~~~ &&  \\
56 && 13860 & ~~&& 756^{\text{x}} && {\bf 1540}(9) & ~~~~~~ && {\bf 924}^{\text{?}}(15), ~ {\bf 1260}^{\text{?}}(11)  \\
58 && 15428 & ~~&& 812^{\text{b}} && 1653^{\text{z}} & ~~~~~~ &&  {\bf 1102}^{\text{?}}(14) \\
62 && 18910 & ~~&& 930^{\text{x}} && {\bf 1891}^{\text{i}}(10) & ~~~~~~ &&   \\
64 && 20832 & ~~&& {\bf 992}^{\text{d}}(21) && 2016^{\text{z}} & ~~~~~~ &&  \\
\hline
\end{array}
$$
The known references in the table for the existence and nonexistence of uniform nested SQSs are as follows:
a - Example~\ref{ex:SQS8uniform}, b - Corollary~\ref{cor:No_min_2_10_mod12}, c - Example~\ref{ex:SQS10uniform},
d - Theorem~\ref{thm:mi_uniform_construct}, e - Example~\ref{ex:ro20}, ~~ f - Example~\ref{ex:ro26}, ~~ g - Lemma~\ref{lem:SQS32},~~
h - Example~\ref{ex:ro38},~~ i - Example~\ref{ex:ro62}, ~~ x - Theorem~\ref{thm:part_uniform}, ~~ z - Theorem~\ref{thm:all_uniform}, ~~
? - no design is known.

How close can we get to a uniform nested SQS
Since it seems to be too difficult to construct a uniform nested SQS and even the set of parameters where they might exist is
too sparse, we will define a type of partition which is almost uniform.
A {\bf \emph{quasi-uniform nested SQS$(v)$}} is a nested SQS$(v)$ in which the difference between the
multiplicity of any two ND-pairs is at most one. The following lemma is an observation from simple counting arguments.

\begin{lemma}
If for an nested SQS$(v)$ the total number of pairs $\frac{v(v-1)(v-2)}{12}$ is divisible by
an integer $k$ such that $\frac{v}{2}(\frac{v}{2} -1) \leq k \leq \binom{v}{2}$, then any quasi-uniform
nested SQS$(v)$ with exactly $k$ ND-pairs is a uniform nested SQS$(v)$.
\end{lemma}
\begin{proof}
If there are exactly $k$ ND-pairs, then let $\mu = \frac{v(v-1)(v-2)}{12k}$. If there are $k_1$ ND-pairs
with multiplicity $\mu +\epsilon$ ($\epsilon$ is an integer) and $k-k_1$ ND-pairs with multiplicity
$\mu +\epsilon +1$, then the total number of pairs, in the partition of each block into pairs, is
$$
\frac{v(v-1)(v-2)}{12} = (\mu+\epsilon) k_1 + (k-k_1)(\mu +\epsilon+ 1)= k \mu  +k \epsilon +k -k_1  = \frac{v(v-1)(v-2)}{12} +k \epsilon +k -k_1.
$$
$k \epsilon +k -k_1=k(\epsilon +1)-k_1$ can be equal to 0 only if $\epsilon =0$ and $k=k_1$ which imply that all the
ND-pairs have multiplicity $\mu$. Thus, any quasi-uniform nested SQS$(v)$ with exactly $k$ ND-pairs is a uniform nested SQS$(v)$.
\end{proof}

Unfortunately, the definition of quasi-uniform nested SQS$(v)$ does not make the family of such designs much richer
(new parameters beyond the ones for uniform nested SQSs)
in terms of known parameters. Nevertheless, we have the following construction for quasi-uniform nested SQSs.
The proof is left to the reader, but an example is presented.

\begin{lemma}
If there exists a complete uniform nested SQS$(v)$,
then there exist quasi-uniform nested SQS$(v)$s which differ
in their number of ND-pairs.
\end{lemma}

\begin{example}
\label{ex:many_quasi}
Starting from the complete uniform nested SQS$(8)$, $(\Z_7 \cup \{\infty\},\cB)$ of Example~\ref{ex:SQS8uniform}, where each one of
the 28 pairs is a ND-pair with multiplicity 1. Map the points of $\Z_7 \cup \{\infty\}$ onto $\Z_8$.
Apply Construction~\ref{cons:2vB} to obtain an SQS$(16)$ on the set of points $\Z_8 \times \Z_2$. This SQS$(16)$ has 112 ND-pairs
with multiplicity 2 (will be called A-set) which form the set
$$
\{ \{ (x,\ell),(y,m) \} ~:~ x,y \in \Z_8, ~ x \neq y, ~ \ell,m \in \Z_2 \} .
$$
This nested SQS$(16)$ also has 8 ND-pairs with multiplicity 7 (will be called B-set) which form the set of pairs
$$
\{ \{ (x,0),(x,1) \} ~:~ x \in \Z_8 \}.
$$
Reducing the 8 ND-pairs
of the B-set to multiplicity 3 will enable to increase 32 ND-pairs of the A-set to multiplicity 3.
For example, we take the 16 quadruples of Type II
$$
\{ \{ (i,0),(i,1),(i+1,0),(i+1,1) \} ~:~ i \in \Z_8 \} \bigcup \{ \{ (i,0),(i,1),(i+2,0),(i+2,1) \} ~:~ i \in \Z_8 \}.
$$
These quadruples in Construction~\ref{cons:2vB} are partitioned into the pairs
$$
\{ [\{ (i,0),(i,1)\},\{(i+1,0),(i+1,1) \}] ~:~ i \in \Z_8 \} \bigcup \{ [\{ (i,0),(i,1)\},\{(i+2,0),(i+2,1) \}] ~:~ i \in \Z_8 \}.
$$
We change their partitions into pairs into other 32 pairs as suggested in Remark~\ref{rem:diff_pairs}.
$$
\{ [\{ (i,0),(i+1,0)\},\{(i,1),(i+1,1) \}] ~:~ i \in \Z_8 \}
$$
and
$$
\{ [\{ (i,0),(i+2,0)\},\{(i,1),(i+2,1) \}] ~:~ i \in \Z_8 \}
$$
Using this partition of each of these block we obtain a quasi-uniform nested SQS with
80 ND-pairs with multiplicity 2 and 40 ND-pairs with multiplicity 3.
We can continue to change the partition of specific blocks into pairs
and obtain the following six quasi-uniform nested  SQS$(16)$: for each $i$, $0 \leq i \leq 5$,
we have $80-3i$ ND-pairs with multiplicity 2, $40+2i$ ND-pairs with multiplicity 3,
and $i$ pairs $\Z_8 \times \Z_2$ which are not ND-pairs. The exact details are left for the reader.

Furthermore, we can have other quasi-uniform partitions. For example, if instead of the partition for Type II
used in Construction~\ref{cons:2vB}, i.e.,
$$
\{ [\{ (x,0),(y,0)\},\{(x,1),(y,1) \}] ~:~ x,y \in \Z_8, ~~ x \neq y \}
$$
we partition these blocks by
$$
\{ [\{ (x,0),(y,1)\},\{(x,1),(y,0) \}] ~:~ x,y \in \Z_8, ~~ x \neq y \}
$$
we obtain a quasi-uniform nested SQS$(16)$ with
$56$ ND-pairs with multiplicity 2, $56$ ND-pairs with multiplicity 3,
and $8$ pairs of $\Z_8 \times \Z_2$ which are not ND-pairs.

\hfill\quad $\blacksquare $
\end{example}

The proof of the following lemma is identical to the proof of Theorem~\ref{thm:mi_uniform_construct}.
\begin{lemma}
If there exists a quasi-uniform nested SQS$(v)$
in which all pairs are ND-pairs, then there exists a nested quasi-uniform SQS$(2v)$ .
\end{lemma}

\section{Conclusions and Problems for Future Research}
\label{sec:conclude}

A partition for each block of an SQS$(v)$, $(Q,\cB)$, into pairs was considered. Each pair which participates
in the nested SQS is a ND-pair. The number of such ND-pairs and their multiplicity
in possible partitions were considered. Several constructions with analysis of the number of ND-pairs
and their multiplicities were presented. The following results on were obtained.

\begin{enumerate}
\item Each element of $Q$ is contained in at least $\frac{v-2}{2}$ ND-pairs (Lemma~\ref{lem:lower_nonzeroPN}). This
bound is attained by Construction~\ref{cons:2vA} (Lemma~\ref{lem:nonzeroP_ConA} and Corollary~\ref{cor:nonzeroP_ConA}).

\item If $v \equiv 2 ~ \text{or} ~ 10 ~(\mmod ~12)$, then each element of $Q$ is contained in at least $\frac{v}{2}$ ND-pairs
(Lemma~\ref{lem:necc_min_NZ_odd}).

\item The number of ND-pairs is at least $\frac{v}{2} \left( \frac{v}{2} -1 \right)$ (Corollary~\ref{cor:lower_nonzeroPN}).
This bound is attained by Construction~\ref{cons:2vA} (Lemma~\ref{lem:nonzeroP_ConA} and Corollary~\ref{cor:nonzeroP_ConA}).

\item If $v \equiv 2 ~ \text{or} ~ 10 ~(\mmod ~12)$, then the number of ND-pairs
is at least $\frac{v^2}{4}$ (Corollary~\ref{cor:necc_min_NZ_odd}).

\item The number of ND-pairs is at most $\binom{v}{2}$ (trivial). This bound is attained by Construction~\ref{cons:2vB}
(Corollary~\ref{cor:all_2v}).

\item The multiplicity of a ND-pair is at most $\frac{v-2}{2}$ (Lemma~\ref{lem:max_multiplicity}).
This bound is attained by Construction~\ref{cons:2vB} (Lemma~\ref{lem:multi_ConB}).

\item The number of ND-pairs with multiplicity $\frac{v-2}{2}$ is at most $\frac{v}{2}$ (Lemma~\ref{lem:maxPairsMax}).
This bound is attained by Construction~\ref{cons:2vB} (Lemma~\ref{lem:all_v} and Corollary~\ref{cor:attain_upper}).

\item At least one ND-pair has multiplicity at most $\frac{v-1}{3}$ (Lemma~\ref{lem:upper_multi_min}).
This bound it attained by minimum uniform nested SQSs (Theorem~\ref{thm:part_uniform}).

\item At least one ND-pair has multiplicity greater or equal $\frac{v-2}{6}$ (Lemma~\ref{lem:upper_multi_max}).
This bound it attained by complete uniform nested SQSs (Theorem~\ref{thm:all_uniform}).

\item Blocks partitions with uniform multiplicities were considered and some examples and nonexistence results on these
partitions were given.
\end{enumerate}

The exposition on this topic leaves more open problems than the ones which were solved.
The following list is just a very small list of problems for future research.
\begin{enumerate}
\item Find a construction of nested complete uniform
SQS$(2^n)$, where $n \geq 3$ is an odd integer.

\item Provide more uniform nested SQSs which are neither complete uniform nor minimal uniform.

\item Find infinite family of quasi-uniform blocks partitions.

\item For a given $v$, $v \equiv 2 ~ \text{or} ~ 10 ~ (\mmod ~ 12)$, $v > 4$, provide constructions of nested
SQS$(v)$ with large multiplicity for each ND-pair.

\item For each value $v \equiv 2 ~ \text{or} ~ 4 ~ (\mmod ~ 6)$ find lower and upper bounds on $\mu_{\text{min}} (v)$ which
is the multiplicity of the ND-pair with the minimum multiplicity in the nested SQS$(v)$, where
the ND-pair with the minimum multiplicity is maximized.

\item Find a method to switch between pairs in the partitions of blocks (as described in Example~\ref{ex:many_quasi})
to obtain quasi-uniform nested SQSs.

\item For any given $v \equiv 2 ~ \text{or} ~ 4 ~ (\mmod ~ 6)$ find a lower bound large as possible on the multiplicity of
the ND-pair with minimum multiplicity for any SQS$(v)$.

\item For a given SQS$(v)$ find a lower bound large as possible on the multiplicity of the ND-pair with minimum multiplicity.

\item Improve on the bounds presented in Section~\ref{sec:large_int} for different families of parameters.

\item Complete the table of the uniform nested SQSs up to SQS$(64)$, using new theorems and computer search, especially for
uniform nested SQSs in which the number of ND-pairs is neither minimum nor maximum.

\end{enumerate}

\section*{Acknowledgement}

The authors would like to thank an anonymous reviewers for their constructive comments that helped to improve
the presentation of the results.

%%%%%%%%%%%%%%%%%%%%%%%%%%%%%%%%%%%%%%%%%%%%%%%%%%%%%%%%%%%%%%%%%%%%%%
%%%%%%%%%%%%%%%%%%%%%%%%%%%%%%%%%%%%%%%%%%%%%%%%%%%%%%%%%%%%%%%%%%%%%%
%%%%%%%%%%%%%%%%%%%%%%%%%%%%%%%%%%%%%%%%%%%%%%%%%%%%%%%%%%%%%%%%%%%%%%
%%%%%%%%%%%%%%%%%%%%%%%%%%%%%%%%%%%%%%%%%%%%%%%%%%%%%%%%%%%%%%%%%%%%%%

\end{document}